\documentclass[11pt]{amsart}
\usepackage{amsmath,amsfonts,amssymb,amsthm,amscd,latexsym,euscript,verbatim}
\usepackage[all]{xy}

\makeatletter
\def\section{\@startsection{section}{1}%
  \z@{1.1\linespacing\@plus\linespacing}{.8\linespacing}%
  {\normalfont\Large\scshape\centering}}
\makeatother

\theoremstyle{plain}

\newtheorem*{conj*}{Root Groups Conjecture}
\newtheorem*{thm1.2}{(1.2) Theorem}
\newtheorem*{thm1.3}{(1.3) Theorem}
\newtheorem*{thm1.4}{(1.4) Theorem}
\newtheorem*{prop*}{Proposition}

\newtheorem{prop}{Proposition}[section]

\newtheorem{thm}[prop]{Theorem}
\newtheorem{cor}[prop]{Corollary}
\newtheorem{lemma}[prop]{Lemma}

\theoremstyle{definition}

\newtheorem*{Def*}{Definition}

\newtheorem{notation}[prop]{Notation}
\newtheorem*{notation*}{Notation}

\newtheorem{remarks}[prop]{Remarks}

\newcommand{\mouf}{\mathbb{M}}

\newcommand{\cala}{\mathcal{A}}
\newcommand{\calb}{\mathcal{B}}

\newcommand{\cald}{\mathcal{D}}

\newcommand{\calm}{\mathcal{M}}

\newcommand{\gt}{\tau}

\newcommand{\nsg}{\trianglelefteq}

\newcommand{\sminus}{\smallsetminus}
\newcommand{\lan}{\langle}
\newcommand{\ran}{\rangle}

\newcommand{\Aut}{{\rm Aut}}

\newcommand{\Inv}{{\rm Inv}}

\newcommand{\widebar}[1]{\overset{\mskip1mu\hrulefill\mskip1mu}{#1}
                \vphantom{#1}}

\numberwithin{equation}{section}

\begin{document}
\title[]{Almost regular involutory automorphisms of uniquely $2$-divisible groups}
\author[Yoav Segev]{Yoav Segev}
\address{
         Department of Mathematics \\
         Ben-Gurion University \\
         Beer-Sheva 84105 \\
         Israel}
\email{yoavs@math.bgu.ac.il}


\begin{abstract}
We prove that a uniquely $2$-divisible group that admits an almost
regular involutory automorphism is solvable.
\end{abstract}

\date{\today}
\maketitle
\section{Introduction}
Recall that an automorphism $\nu$ of a group $H$ is
called {\it involutory} if $\nu\ne id$ and $\nu^2=id$.
The automorphism $\nu$ is called {\it almost regular}, if $C_H(\nu)$ is finite.
Recall that a group $U$ is {\it uniquely $2$-divisible}
if for each $u\in U$ there exists a unique $v\in U$ such that $v^2=u$.
Note that in particular a uniquely $2$-divisible group contains no involutions
(i.e.~elements of order $2$).

The purpose of this note is to use the techniques
introduced in the impressive paper \cite{Sh} of Shunkov,
where he proves that a periodic group that admits an almost regular
involutory automorphism is virtually solvable (i.e.~it has
a solvable subgroup of finite index).  We prove.
%
%
\begin{thm}\label{thm main}
Let $U$ be a uniquely $2$-divisible group.  If $U$
admits an involutory almost regular automorphism,
then $U$ is solvable.
\end{thm}

\noindent
Our main motivation for dealing with automorphisms of uniquely $2$-divisible
groups comes from questions about the root groups of special Moufang sets,
and those tend to be uniquely $2$-divisible, see, e.g., \cite{S}.
Indeed, using Theorem \ref{thm main} it immediately follows that

%
\begin{cor}
Let $\mouf(U,\gt)$ be a special Moufang set.  If the Hua
subgroup contains an involution $\nu$ such that $C_U(\nu)$
is finite, then $U$ is abelian.
\end{cor}
\begin{proof}
If $U$ contains involutions, then $U$ is abelian by \cite[Theorem 5.5, p.~782]{DST}.
If $U$ does not contain involutions, then by \cite[Proposition 4.6, p.~5840]{DS},
$U$ is uniquely $2$-divisible, and then by Theorem \ref{thm main} and by the main theorem of \cite{SW},
$U$ is abelian.
\end{proof}

The proof of Theorem \ref{thm main} is obtained as follows.
First note that if $U$ is finite, then $U$ has odd order,
so by the Feit-Thompson theorem $U$ is solvable.
Hence we may assume that $U$ is infinite.

We let $A$ be a maximal abelian subgroup of $U$ (with respect to inclusion) inverted by $\nu$,
(i.e.~each element of $A$ is inverted by $\nu$).  In Lemma \ref{lem exofA}(2)
we show that we can take $A$ to be infinite.  We then show that
for elements $u_1,\dots, u_n\in U$, the involutions
$u_1\nu u_1^{-1},\dots, u_n\nu u_n^{-1}$ in the semi-direct product
$U\rtimes \lan \nu\ran$ invert a subgroup $D\le A$ with $|A:D|<\infty$
(Proposition \ref{prop A:Au}).  The next step is to show that
$C_U(D)/D$ is finite and solvable (Lemma \ref{lem CUD}).
Since $K:=\lan \nu u_1\nu u_1^{-1},\dots, \nu u_n\nu u_n^{-1}\ran\le C_U(D)$,
the subgroup $K$ is solvable and $K/Z(K)$ is finite.

Next let $S:=\{x\in U\mid x^{\nu}=x^{-1}\}$.
It is easy to see that an element $y\in U$ is in $S$ iff $y=\nu u\nu u^{-1}$,
for some $u\in U$, so by the above each finitely generated subgroup $H$
of $R:=\lan S\ran$ is solvable and satisfies: $H/Z(H)$ is finite.
Hence $R'$ is periodic (Proposition \ref{prop <S>}).
Using the above mentioned result of Shunkov, we see that $R'$
is solvable, so $R$ is solvable.

As is well known (see \cite{K})
$U=R C_U(\nu)$ and $R\nsg U$.  Since $C_U(\nu)$ is finite and uniquely $2$-divisible
it has odd order.  By the Feit-Thomson theorem, $C_U(\nu)$
is solvable and this at last shows that $U$
is solvable.

We remark that it is possible that with the aid of the Theorem
on page 286 of \cite{HM},
one can get even more delicate information on $U$, but we do not need
that, so we do not pursue this avenue further. 

\section{Notation and preliminary results}\label{sec not}

\begin{notation}\label{not A}
\begin{enumerate}
\item
Throughout this note $U$ is an infinite uniquely $2$-divisible group and
$\nu\in\Aut(U)$ is an involutory automorphism  which is almost regular.

\item
We denote by $G$ the semi-direct product of $U$ by $\nu$
and we indentify $U$ and $\nu$ with their images in $G$.
We let $\Inv(G)$ denote the set of involutions of $G$.

\item
We let $S:=\{x\in U\mid x^{\nu}=x^{-1}\}$.

\item
The  letter $A$ always denotes a fixed infinite maximal
(with respect to inclusion) abelian subgroup of $U$ which is
inverted by $\nu$ (i.e.~all of whose elements are
inverted by  $\nu$).  The existence of $A$ is guaranteed by
Lemma \ref{lem exofA}(2) and by Zorn's lemma.

\item
For each $u\in U$ we denote by $A_u$ the subgroup of $A$
inverted by $u\nu u^{-1}$.
\end{enumerate}
\end{notation}

\begin{remarks}\label{rem basic}
\begin{enumerate}
\item
Notice that $A$ is uniquely $2$-divisible.

\item
Note also that for any $u\in U$, the subgroup $A_u$ is
uniquely $2$-divisible.

\item
It is easy to check that $S=\{\nu\nu^x\mid x\in U\}$.
\end{enumerate}
\end{remarks}

\begin{lemma}[\cite{N}, Lemma 4.1, p.~239]\label{lem neumann}
Let the group $H$ be the union of finitely many, let us say $n$,
cosets of subgroups $C_1, C_2,\dots,  C_n:$
\[
H=\textstyle{\bigcup}_{i=1}^n C_ig_i,
\]
Then the index of (at least) one of these subgroups in $H$ does not exceed $n$.
\end{lemma}

\begin{cor}\label{cor neumann}
Let the group $H$ be the union of finitely many, let us say $n$,
subsets $S_1, S_2,\dots,  S_n:$
\[
H=\textstyle{\bigcup}_{i=1}^n S_i.
\]
For each $i$ set $C_i:=\lan ab^{-1}\mid a, b\in S_i\ran$.  Then
the index of (at least) one of the  subgroups $C_1,\dots, C_n$ in $H$ does not exceed $n$.
\end{cor}
\begin{proof}
For each $i=1,\dots, n$, pick an arbitrary $g_i\in S_i$.  Notice that $S_i\subseteq C_ig_i$ for
all $i$, so $H=\bigcup_{i=1}^nC_ig_i$ and the Corollary follows from Lemma \ref{lem neumann}.
\end{proof}

\begin{lemma}
\begin{enumerate}
\item
All involutions in $G$ are conjugate;

\item
$S=\{\nu \gt\mid \gt\in\Inv(G)\}$.
\end{enumerate}
\end{lemma}
\begin{proof}
Let $\gt\in\Inv(G)$.  Then $\gt =x\nu,$ for some $x\in U$.
Since $\gt$ is an involution $x\in S$.  Let $y\in U$ be the unique
element with $y^2=x$, then $y\in S$ and $\gt=x\nu =y^2\nu=y\nu y^{-1}$.
This shows (1).  Part (2) is Remark \ref{rem basic}(3).
\end{proof}

\begin{lemma}\label{lem basic}
Let $D$ be an abelian uniquely $2$-divisible subgroup of $U$.  Then
\begin{enumerate}
\item
$C_U(D)/D$ is a uniquely $2$-divisible group.

\item
If $D$ is inverted by $\nu,$
then $\nu D$ is an almost regular
involutory automorphism of $C_U(D)/D$.

\item
Assume that $D$ is inverted by $\nu$ and let
$E/D$ be a subgroup of $C_U(D)/D$
which is inverted by $\nu D$.  Then
$E$ is inverted by $\nu$, so, in particular, $E$ is
abelian.
\end{enumerate}
\end{lemma}
 \begin{proof}
 (1):\quad
 Set $C:=C_U(D)$.  Assume that $a, b\in C$ and $a^2D=b^2D$.
 Let $x, y\in D$ with $a^2x=b^2y$ and let
 $u, v\in D$ with $u^2=x$ and $v^2=y$.
 Then $a^2u^2=b^2v^2$ and since $a, b$
 commute with $u, v$ we see that $(au)^2=(bv)^2$,
 hence $au=bv$ so $aD=bD$.

 Furthermore let $aD\in C/D$.  Let $b\in U$ with $b^2=a$.
 Then $b\in C$ and $bD$ is the square root of $aD$ in $C/D$.

 \medskip

 \noindent
 (2):\quad
Clearly $\nu D$ is an involutory automorphism of  $C/D$
(acting via conjugation).  Assume that $aD\in C/D$
 centralizes $\nu D$.  Then $\nu^a =\nu d$, for some $d\in D$.
 Let $x\in D$ with $x^2=d$.  Then $\nu$ inverts $x$ and we see
 that $\nu^a=\nu^x$ and $ax^{-1}\in C_U(\nu)$.  It follows that
 $C_{C/D}(\nu D)=C_C(\nu)D/D,$ and since $\nu$ is almost regular,
 so is $\nu D$.
 \medskip

 \noindent
 (3):\quad
 Let $xD\in C/D$ be an element inverted by $\nu D$.  Then
 $x^{\nu}=x^{-1}d$, for some $d\in D$, and conjugating by $\nu$
 we see that $x=x^{-\nu}d^{-1}$ which implies that $x^{\nu}=x^{-1}d^{-1}$.
 Thus $d=d^{-1}$ so $d=1$.

 Now let $e\in E$.  Then, by hypothesis, $eD$ is inverted by $\nu D$,
 so $e^{\nu}=e^{-1}$.
 \end{proof}

\section{The proof of  Theorem \ref{thm main}}
%
%
%

\begin{lemma}\label{lem exofA}
Let $D$ be an abelian subgroup of $U$ (we allow $D=1$) such
that $D$ is inverted by $\nu$ and such that $C_U(D)$ is infinite.
Assume that
\[
(S\cap C_U(D))\sminus D\ne\emptyset.
\]
Then
\begin{enumerate}
\item
there exists an element $w\in C_U(D)\sminus D$ which is inverted
by $\nu$ and such that $C_U(\lan D, w\ran)$ is infinite;

\item
There exists an infinite abelian subgroup of $U$ which is inverted by $\nu$.
\end{enumerate}
\end{lemma}
\begin{proof}
(1):\quad
Set $V:=C_U(D)$.  Then $V$ is an infinite  uniquely $2$-divisible group, and
$\nu$ acts on $V$,
so without loss we may assume that $U=V$ and that $D\le Z(U)$.

Pick $b\in S\sminus D$ (note that $b$ exists by hypothesis),
and write $b=\nu\gt$ with $\gt\in\Inv(G)$.
Let
\[
u\in U \text{ with } u^{-2}=\nu \gt,
\]
and note that since $u$ is inverted by both $\nu$ and $\gt,$ we have
\[
\nu=\gt^u.
\]
We now find an element $h\in C_U(\gt)$ such that $hu$ is inverted by infinitely
many involutions of $G$.  Note that $hu\notin D$; indeed, if $h=1$, then $hu=u$
and since $b\notin D$ also $u\notin D$. Otherwise if $hu\in D$ and $h\ne 1,$ then
\[
u^{-1}h^{-1}=(hu)^{\gt}=h^{\gt}u^{\gt}=hu^{-1},
\]
and it follows that $u$ inverts $h$ which is not possible in
a uniquely $2$-divisible group.

Since all involutions
in $G$ are conjugate, conjugating $hu$ by an appropriate
element we may assume that $\nu$ inverts $hu$
and since $hu$ is inverted by infinitely many involutions
we see that $C_U(hu)$ is infinite and taking $w=hu$ we are done.

It remains to show the existence of $h$.  For each $a\in S$, let
\[
s_a:=\nu \gt^{a} \text{ and } \ell_a^{-2}=s_a.
\]
It is easy to check that since $\ell_a$ is inverted by $\nu$ and $\gt^a$, we have
$\gt^{a\ell_a}=\nu$.  Hence
\[
\gt^{a\ell_a}=\gt^u,\text{ and hence }h_a:=a\ell_au^{-1}\in C_U(\gt).
\]
It follows that $\ell_a=a^{-1}h_au$.  Since both $\ell_a$ and $a$ are inverted
by $\nu$ we get after conjugating by $\nu$ that $\ell_a^{-1}=a(h_a u)^{\nu}=(h_au)^{-1}a$.
Notice now that $a\nu\in\Inv(G)$
and it follows that
\[
(h_au)^{a\nu}=(h_au)^{-1}.
\]
By hypothesis the set $\{h_a\mid a\in S\}$ is finite since it is
contained in $C_U(\gt)$.  Further, the set $S$ is infinite.
This implies the existence of $h\in C_U(\gt)$ such
that the number of involutions $a\nu$ that invert $hu$ is infinite.
This proves (1).
\medskip

\noindent
(2):\quad
If $D$ is finite and $C_U(D)$ is infinite, then $(S\cap C_U(D))\sminus D\ne\emptyset$.
Hence part (2) follows from (1) by starting with $D=1$ and iterating the process as long as the subgroup
$\lan D, w\ran$ is finite.
\end{proof}

\begin{lemma}\label{lem finiteB}
Let $B$ be a finitely generated abelian subgroup of $U$ which is inverted by $\nu$.
Then $A$ contains a subgroup $A_1$ of finite index such that $\lan A_1, B\ran$
is abelian.
\end{lemma}
\begin{proof}
Let $\calb$ be a finite set of generators for $B$ and set
$A_1:=\bigcap_{b\in \calb}A_b$.  By Proposition \ref{prop A:Au} and since $\calb$
is finite $|A:A_1|<\infty$.  Further, for each $b\in \calb$, $\nu$ and $b\nu b^{-1}$ invert
$A_1$, so $b^2=b\nu b^{-1}\nu\in C_U(A_1)$ (recall that $\nu$ inverts $b$).
Since $U$ is uniquely $2$-divisible, $b\in C_U(A_1)$.  Hence $\calb\le C_U(A_1)$ and the lemma holds.
\end{proof}

\begin{lemma}\label{lem CUD}
Let $D$ be a uniquely $2$-divisible subgroup of $A$ of finite index.  Then
$C_U(D)/D$ is finite and solvable.
\end{lemma}
\begin{proof}
Set $C:=C_U(D)$ and $\widebar C:=C/D$.  Assume that $\widebar C$ is infinite.  By Lemma \ref{lem basic}(1),
$\widebar C$ is uniquely $2$-divisible,
and by hypothesis $\widebar A:=A/D$ is a finite subgroup of $\widebar C$.

Let $\widebar \cala$ be an infinite maximal abelian subgroup
of $\widebar C$ inverted by $\nu D$.  The existence of $\widebar \cala$
is guaranteed by Lemma \ref{lem basic}(2) and  by Lemma \ref{lem exofA}(2) (with $\widebar C$ in place of $U$).
By Lemma \ref{lem finiteB} (with $\widebar C$ in place of $U$ and
$\widebar A$ in place of $B$), there exists an finite index $\widebar {\cala_1}\le \widebar \cala$
such that $\widebar {\cala_2}:=\lan\widebar {\cala_1}, \widebar A\ran$ is abelian.
Note that $\widebar {\cala_2}$ is inverted by $\nu D$, so by Lemma \ref{lem basic}(3),
the inverse image $\cala_2$ of $\widebar {\cala_2}$ in $C_U(D)$ is an abelian subgroup
inverted by $\nu$.  Clearly $\cala_2$ properly contains $A$.  This contradicts the maximality
of $A$ and shows that $\widebar C$ is finite.

Let $\cald\le C$ be a maximal central subgroup of $C$ which
is inverted by $\nu$.  Of course $\cald\ge D$.  Further, it
is clear that $\cald$ is a uniquely $2$-divisible group.
Suppose $t\cald$ is an involution in $C/\cald$.  Then $t^2\in\cald$,
so also $t\in\cald$ and we see that $C/\cald$ has odd order.
By the Feit-Thompson theorem, $C/\cald$ is solvable, and the proof
of the lemma is complete. 
\end{proof}

\begin{lemma}\label{lem xs}
Let $x\in U$ and let $s\in U$ be the unique element
such that $s^{-2}=\nu x^{-1}\nu x$.  Then $xs\in C_U(\nu)$.
\end{lemma}
\begin{proof}
Notice that $s$ is inverted by $\nu$ and $\nu^x$.  Hence
\[
1=s^2\nu \nu^x=\nu s^{-2}\nu^x=\nu s^{-1}\nu^xs,
\]
so the lemma holds.
\end{proof}

\begin{prop}\label{prop A:Au}
Let $A$ be as in Notation \ref{not A}(4) and let $u\in U$.
Let $A_u$ be as in Notation \ref{not A}(5).  Then $|A:A_u|<\infty$.
\end{prop}
\begin{proof}
Fix $a\in A$ and  consider the element
\[
\nu \nu^{au},\quad a\in A.
\]
This element is in $U$.  Let $s\in U$ with $s^{-2}=\nu\nu^{au}$.
By Lemma \ref{lem xs} we get that
%
%
\begin{equation}\label{eq va}
v_a:=aus\in C_U(\nu).
\end{equation}
%
%
Now set
\[
\calm_a:=\{b\in A\mid v_b=v_a\}.
\]
Notice that since $|C_U(\nu)|<\infty,$
%
%
\begin{equation}\label{eq calma}
\text{the set }\{\calm_c\mid c\in A\}\text{ is finite and }A=\textstyle{\bigcup}_{c\in A}\calm_c.
\end{equation}
By equation \eqref{eq va} we get $s^{-1}=v_a^{-1}au$ and conjugating
by $\nu$ noticing that $\nu$ inverts $a$ and $s$ and centralizes $v_a$ we see that $s^{-1}=u^{-\nu}av_a$.
So we get the equality
\[
v_a^{-1}au=u^{-\nu}av_a,
\]
from which it follows that
%
%
\begin{equation}\label{eq seq}
u^{-1}\nu bv_au^{-1}=\nu v_a^{-1}b,\quad \forall b\in\calm_a.
\end{equation}
%
%
Let $c\in \calm_a$, then as in equation \eqref{eq seq} we get that
$u^{-1}\nu cv_au^{-1}=\nu v_a^{-1}c$ and this together with equation \eqref{eq seq}
yields
\[
uv_a^{-1}c^{-1}bv_au^{-1}=c^{-1}b,\quad \forall b, c\in\calm_a.
\]
Since $\nu$ inverts $c^{-1}b\in A$, it follows that $uv_a^{-1}\nu v_au^{-1}=u\nu u^{-1}$
inverts $c^{-1}b$.  We thus can conclude that
%
%
\begin{equation}\label{eq unuu-1 inverts}
u\nu u^{-1}\text{ inverts }\lan bc^{-1}\mid b, c\in\calm_a\ran,\quad\forall a\in A.
\end{equation}
%
%
By equation \eqref{eq calma} and by Corollary \ref{cor neumann}
one of the groups $\lan bc^{-1}\mid b, c\in\calm_a\ran$ has finite index
in $A$, so $|A:A_u|<\infty$ as asserted.
\end{proof}

\begin{prop}\label{prop <S>}
Let $R:=\lan S\ran$, then
 \begin{enumerate}
 \item\
 $R'$ is a periodic group;
 
 \item
 $R$ is solvable.
 \end{enumerate}
\end{prop}
\begin{proof}
(1):\quad
We first show that 
for elements $u_1,\dots, u_n\in U$
the subgroup $K:=\lan \nu u_1\nu u_1^{-1},\dots, \nu u_n\nu u_n^{-1}\ran$
is solvable, and $K/Z(K)$ is finite.  By Remark \ref{rem basic}(3),
this will show that
\smallskip

\begin{itemize}
\item[(*)]
if $H$ is a f.g.~subgroup of $R,$ then $H$ is solvable,\\
and $H/Z(H)$ is finite.
\end{itemize}
\smallskip

Let $D:=\bigcap_{i=1}^n A_{u_i}$.  By the definition
of $A_{u_i}$ and by Proposition \ref{prop A:Au}, $|A:D|<\infty$ and $D$ is inverted
by $\nu, u_1\nu u_1^{-1},\dots, u_n\nu u_n^{-1}$.
Also, by Remark \ref{rem basic}(2), $D$ is uniquely $2$-divisible.
By Lemma \ref{lem CUD}, $C_U(D)/D$ is finite and solvable, so since
$K\le C_U(D)$, we see that $K/Z(K)$ is finite and solvable.
Hence (*) holds.

Next let $g\in R'$.  Then there exists a finitely generated subgroup
$H$ of $R$ such that $g\in H'$.  By (*) and by \cite[(33.9), p.~168]{A},
$H'$ is finite, so the order of $g$ is finite.
This completes the proof of part (1).
\medskip

\noindent
(2):\quad
By (1), $R'$ is a periodic group and since $R$ is $\nu$-invariant,
$\nu$ is an almost regular automorphism of $R'$.  By the main result
of Shunkov in \cite{Sh}, $R'$ is virtually solvable.  But by (*),
$R'$ is also locally solvable, so this shows that $R'$ is solvable and
hence so is $R$.
\end{proof}

\medskip

\noindent
\begin{proof}[Proof of Theorem \ref{thm main}.]
By Proposition \ref{prop <S>}, $\lan S\ran$
is solvable.
By \cite[(3.4), p.~281]{K} (see also
\cite[Lemma 2.1(1) and Lemma 2.2(1)]{S}), $U=\lan S\ran C_U(\nu)$
and $\lan S\ran\nsg U$.
Since $C_U(\nu)$ is a finite uniquely $2$-divisible group,
so it has odd order.  By the Feit-Thompson theorem it is solvable.
Hence $U$ is solvable.
\end{proof}

\subsection*{Acknowledgment.}  I would like to
thank Pavel Shumyatsky for several fruitful
email exchanges.


\end{document}